\documentclass[11pt]{article}

\usepackage{amssymb,amsmath,tikz}

\usepackage{hyperref}
\usepackage{url}
\usepackage{amsthm}

\usepackage{tikz-qtree}

\newcommand{\ZZ}{\mathbb{Z}}
\newcommand{\QQ}{\mathbb{Q}}
\newcommand{\PP}{\mathsf{P}}

\newcommand{\p}{\mathsf{p}}

\renewcommand{\v}{\mathsf{v}}
\newcommand{\w}{\mathsf{w}}

\renewcommand{\u}{\mathsf{u}}

\newcommand{\one}{\mathsf{1}}

\newcommand{\linkage}{{\mathcal{L}}}
\newcommand{\ch}{{\mathcal{C}}}
\newcommand{\VV}{{\mathcal{V}}}

\def\character{\mathrm{char}}

\newcommand{\entry}[2]{#1(#2)}

\newcommand{\even}[1]{{\delta_{\mathrm{even}}(#1)}}

\newcommand{\powerset}[1]{{\mathcal{P}(#1)}}
\newcommand{\strictsubsets}[2]{{\mathrm{des}_{#2}(#1)}}
\newcommand{\subsets}[2]{{\overline{\mathrm{des}}_{#2}(#1)}}
\newcommand{\supersets}[2]{{\mathrm{anc}_{#2}(#1)}}

\newcommand{\disjointsets}[2]{{\mathrm{disj}_{#2}(#1)}}

\newcommand{\rank}{\rho}
\newcommand{\base}{\alpha}
\newcommand{\coef}{c}
\newcommand{\coeflink}{d}

\newcommand{\calS}{\mathcal{S}}
\newcommand{\calT}{\mathcal{T}}

\newcommand{\setvec}[1]{{\v_{#1}}}
\newcommand{\outset}[1]{\widehat{#1}}

\newcommand{\coefsum}[2]{C(#1,#2)}
\newcommand{\coefsumlink}[2]{D(#1,#2)}

\newcommand{\construct}[2]{K_{#2}(#1)}

\newcommand{\bitstrings}[1]{X_{#1}}
\newcommand{\comp}[1]{ - #1}

\newtheorem{theorem}{Theorem}[section]
\newtheorem{lemma}[theorem]{Lemma}

\newtheorem{corollary}[theorem]{Corollary}
\newtheorem{definition}[theorem]{Definition}

\newtheorem*{nonum}{}

\title{The Voter Basis and the Admissibility of Tree
Characters}

\author{Andrew Beveridge\footnote{Department of Mathematics, Statistics and Computer Science, Macalester College, Saint Paul, MN 55105} ~and Ian Calaway\footnote{Department of Economics, Stanford University, Stanford, CA 94305. \texttt{icalaway@stanford.edu}}}

\begin{document}

\maketitle

%

\begin{abstract}
When making simultaneous decisions, our preference for the outcomes on one subset can depend on the outcomes on a disjoint subset. In referendum elections, this gives rise to the separability problem, where a voter must predict the outcome of one proposal when casting their vote on another. A set $S \subset [n]$ is separable for preference order $\succeq$ when our ranking of outcomes on $S$ is independent of outcomes on its complement $[n]-S$. The admissibility problem asks which characters $\ch \subset \powerset{[n]}$ can arise as the collection of separable subsets for some preference order. 
We introduce a linear algebraic technique to construct preference orders with desired characters. Each vector in our $2^n$-dimensional voter basis induces a simple preference ordering with nice separability properties. Given any collection $\ch \subset \powerset{[n]}$ whose subset lattice has a tree structure, we use the voter basis to construct a preference order with character $\ch$. 


\end{abstract}


\section{Introduction}

Ranking sets of alternatives  has received widespread attention in the social sciences \cite{barbera}. For economists, interdependent consumer preferences  provide insight into which goods are complements or substitutes. Such information  helps vendors to choose inventory, or to design marketing materials and store layouts that  encourage cluster purchasing of interrelated items \cite{shocker}. 
Meanwhile, understanding the implications of preference interdependencies is critical in social choice theory \cite{lacy,hodge}. Interrelated preferences can result in problematic outcomes for referendum elections. A voter must cast their votes for multiple simultaneous proposals, so they are forced to guess the overall outcome when expressing their preferences. This encourages strategic voting, rather than expressing true preferences.  We contribute to the study of the \emph{admissibility problem} \cite{hodge3}, which seeks to characterize the achievable patterns of ranking interdependencies. 

A preference relation is an ordering $\succeq$ on the power set $\powerset{[n]}$. The relation $A \succ B$ means that we prefer outcome $A$ to outcome $B$, while $A \sim B$ means that we are indifferent between these two outcomes. 
In economics, the ground set $[n]=\{1,2,\ldots,  n\}$ could be a set of goods available in a store, where the outcome $A$ corresponds to a consumer's purchases on a particular shopping excursion. In social choice theory,  the ground set could be
a set of proposals in a referendum election, where the outcome $A$ corresponds to the ``yes'' votes of a given voter.

It can be convenient to replace the power set $\powerset{[n]}$ with $n$-dimensional binary space $\bitstrings{[n]} = \ZZ_2^n$, where the subset $A$ corresponds to the binary word $x = x_1 x_2 \cdots x_n$ such that $x_i =1$ when $i \in A$ (and $x_i =0$ otherwise). 
For example, the preference order
$$
\{ 1,2\} \succ \{1 \} \succ \emptyset \succ \{2\}, 
$$
 can  be written in binary notation as
\begin{equation}
 \label{eqn:pref-ex}
11 \succ 10 \succ 00 \succ 01.
\end{equation} 
The bitstring formulation 
is particularly suited for studying  multiple-criteria binary decision processes. Here are two situations where the above ranking corresponds to a reasonable individual's preferences.
First, consider a two-item shopping trip for burgers and buns. The ranking of equation \eqref{eqn:pref-ex} corresponds to the  preference order
\begin{center}
burgers and buns $\succ$ only burgers $\succ$ neither  $\succ$ only buns.
\end{center}
Second, this preference order could reflect a voter's preference for the outcome of a city referendum election, where the first proposal is whether to sponsor a new professional sports team and the second proposal is whether to build a new stadium.  This voter's least preferred outcome would be to build a new stadium without bringing a  team to play there.

We now formulate our notion of dependence and independence among outcomes. 
For sets $S,T \subset [n]$, let $T - S = \{ i : i \in T \mbox{ and } i \notin S \}$ denote the relative complement of $S$  in $T$. When $T=[n]$, we use $\comp{S}$ to denote the complement of $S$.
A subset $S \subset [n]$ is \emph{separable} with respect to preference order $\succeq$ when the individual's preferences for outcomes on $S$ are independent of the outcomes on $\comp{S}$. 
\begin{definition}
\label{def:separable}
The set $S \subset [n]$ is \emph{separable} with respect to $\succeq$ when for every $X,Y \subset S$ and every $Z \in \comp{S}$,
\begin{equation}
\label{eqn:sep-set}
X \succeq Y \Longleftrightarrow X \cup Z \succeq Y \cup Z.
\end{equation}
Otherwise, the set $S$ is \emph{nonseparable}.
\end{definition}
In other words, a set $S \subset [n]$ is \emph{separable} when preferences for outcomes on $S$ are independent of outcomes on $\comp{S}$; otherwise $S$ is \emph{nonseparable}.

In practice, nonseparable preferences can be problematic. Brams et al.~\cite{brams} showed that nonseparable preferences can lead to an election paradox where no voter's ballot matches the final outcome. Lacy and Niou  \cite{lacy} went further to show that the final outcome could be every voter's least favored result. Given their potential consequences, a mathematical understanding of the complexities of separable preferences is in order.

A bitstring formulation of the separability condition \eqref{eqn:sep-set} will be particularly useful. 
A \emph{partial outcome} $x_S$ is a bitstring on $S$.  
We can write any  bitstring $x$ as the concatenation of partial outcomes $x=x_S x_{\comp{S}}$, where we allow ourselves to reorder the criteria as convenient. Let $0_T$ denote the all-zero outcome on $T \subset [n]$. The set $S$ is separable with respect to $\succeq$ when for every $x_S, y_S$ and $v_{\comp{S}}$,
\begin{equation}
\label{eqn:sep-outcome}
x_S 0_{\comp{S}} \succeq y_S 0_{\comp{S}} \Longleftrightarrow x_S v_{\comp{S}} \succeq y_S v_{\comp{S}}.
\end{equation}
In other words, the ranking of partial outcomes on $S$ is independent of the outcome on $\comp{S}$.

Note that $\emptyset$ and $[n]$ are vacuously separable for any preference ordering. Returning to our
 burgers-and-buns example  \eqref{eqn:pref-ex}, the shopper's preference for burgers is separable. Indeed, conditioning on the two possible outcomes for the  second item (buns), we have $11 \succ 01$ and $10 \succ 00$, which means that regardless of whether the store is out of buns, the shopper prefers buying burgers over not buying burgers. Meanwhile, her preference for buns is non-separable. Conditioning on the outcome for first item (burgers), we have $11 \succ 10$ and $00 \succ 01$.   If burgers are in stock, then she prefers to buy buns. However, if she cannot buy burgers, then her bun preference  flips: she would prefer buying nothing over buying buns alone.

\subsection{The Admissibility Problem}

The collection 
\begin{equation}
\label{eqn:char-def}
\character(\succeq)= \{ S \subset [n] : S \mbox{ is separable with respect to } \succeq \}
\end{equation}
 is called the \emph{character} of $\succeq$. The character for the preference order in equation \eqref{eqn:pref-ex} is
 $\character(\succeq) = \{\emptyset, \{ 1 \}, \{1,2\} \}$. 
 When $\character(\succeq)={\mathcal{
P}}([n])$, we say that $\succeq$ is \emph{completely separable}, and when $\character(\succeq)=\{\emptyset, [n]\}$, we say that $\succeq$ is \emph{completely nonseparable}. Both completely separable and completely non-separable preferences have been constructed for arbitrary $n$.
In particular, Hodge and TerHaar \cite{hodge3} showed that as $n \rightarrow \infty$, the probability that a randomly chosen preference order is completely non-separable tends to 1.

Completely separable preferences appear in the literature under various names \cite{oeis}. Indeed,  when every subset of $[n]$ is separable, we have a preference relation that satisfies
 de Finetti's axiom \cite{definetti}, namely that
\begin{equation*}
A \preceq B \Longleftrightarrow A \cup C \preceq B \cup C \quad \mbox{when} \quad (A \cup B) \cap C = \emptyset.
\end{equation*}
Maclagan referred to orders satisfying de Finetti's axiom as  \emph{boolean term orders}  and studied their combinatorial and geometric properties \cite{maclagan}. In probability theory (where they enjoy applications in economics), they are known as  \emph{comparative probability orders} and  \emph{linear qualitative probabilities} \cite{kraft,fishburn,slinko}. For more on the structure and enumeration of completely separable preferences, see \cite{maclagan,bradley,christian}.

Generalizing the study of complete separability, Hodge and TerHaar \cite{hodge3} posed the \emph{admissibility problem}: which families of subsets $\ch \subset \powerset{[n]}$ are \emph{admissible}, meaning that there is  a preference order $\succeq$ on $\powerset{[n]}$ with $\character({\succeq}) = \ch$. 
Generating a wide range of separability patterns is valuable for both social choice theory and economics. Simulation of electorates with diverse separabilities is essential for measuring the  impact of nonseparability, and to test possible mitigation strategies. Simulation of  trading economies with nuanced dependency patterns within and between sectors is essential for understanding potential cascade effects between markets.

We survey recent work on  \emph{admissible characters}. 
An admissible character  must contain both $\emptyset$ and $[n]$, since each of these sets  trivially satisfies equation \eqref{eqn:sep-outcome}.   
Bradley, Hodge and Kilgour \cite{bradley} proved that  admissible characters are  closed under intersections. Hodge and TerHaar   \cite{hodge3} proved that this closure condition is sufficient for $n\leq3$, but not for larger $n$. When $n=4$ there is exactly one inadmissible character satisfying this intersection closure condition:
$$\{\emptyset,\{1, 2\},\{2\},\{2, 3\},\{3\},\{3, 4\},\{1, 2, 3, 4\}\}.$$
Hodge, Krines and Lahr \cite{hodge4} used preseparable extensions to construct certain classes of characters by recursively stitching together total orders on disjoint ground sets. For each of these characters $\ch$, there is at least one proper, nonempty $S$ such that both $S$ and $-S$ are in $\ch$.  Recently, Bjorkman, Gravelle and Hodge \cite{bjorkman} used Hamilton paths on the hypercube to generate orders called cubic preferences. The characters that they construct consist of nested subsets 
$\emptyset \subset S_1 \subset S_2 \cdots \subset S_k \subset [n]$.


Herein, we extend the landscape of constructible characters by introducing a linear algebraic framework and using it to generate preferences with a tree structure of separable sets. Our admissible families are sublattices of the boolean lattice $\mathcal{B}_n$ of subsets of $[n]$ ordered by inclusion,  so we first recall some helpful poset terminology.
Set containment induces a  partial order on any character $\ch$. For $A,B \in \ch$, we define $A \prec B$ when $A \subsetneq B$. We say that $B$ \emph{covers} $A$ when $A \prec B$ and there is no $C \in \ch$ such that $A \prec C \prec B$.  The \emph{Hasse diagram} of a poset is an acyclic directed graph that has an edge from vertex $A$ to $B$ whenever $B$ covers $A$. The graph layout is drawn so that $B$ appears above all sets that it covers, so that all edges are oriented upwards. The unique maximal element is $[n] \in \ch$ and unique minimal element is $\emptyset \in \ch$.

%
%

This brings us to our main result.

\begin{definition}
\label{def:tree-char}
A \emph{tree character} $\ch$ is a collection of subsets of $[n]$ such that the Hasse diagram of $\ch - \{ \emptyset \}$ is a tree rooted at $[n]$. In other words, if  $A$,$B\in \ch$ then one of the following is true: 
$$A=B, A\subsetneq B, A\supsetneq B, \mbox{ or } A\cap B=\emptyset.$$ 
\end{definition}

\begin{theorem}
\label{thm:tree-char}
Every tree character is admissible.
\end{theorem}

Our method for constructing preference orderings with tree characters is particularly valuable for economic trade applications. Goods are frequently organized hierarchically, either by sector categorization or via clustering methods. Running  trade simulations with  preference dependencies are drawn from these hierarchies would help to measure the robustness of models where dependencies between goods cause cascading effects of economics shocks.

\subsection{The Preference Space $\PP^n$}
\label{sec:pref-space}

Our proof of Theorem  \ref{thm:tree-char} uses linear algebraic methods to construct a preference order with the desired character. We introduce the $2^n$-dimensional \emph{voter basis}, whose vectors induce preference orderings with nice separability properties. We believe that this flexible voter basis has the potential to significantly expand the set of known admissible characters, and perhaps more importantly,  to provide insight into the structure of completely separable preferences (aka comparitve probability orderings). 

We adopt the linear algebraic viewpoint by converting a preference order into a $2^n$-dimensional vector. 
Consider the \emph{preference space} $\PP^n \cong \QQ^{2^n}$ whose basis vectors are indexed by bitstrings from $\ZZ_2^n$ (or equivalently, by subsets of $[n]$).
We view a preference vector in $\PP^n$ as a utility function on  outcomes, where a higher value corresponds to a more preferred outcome.  Starting with a preference order $\succeq$, we construct the preference vector $\v_{\succeq}$ by setting the least preferred entry to 0 and then assigning the other utilities incrementally.
For example, the preference order in equation \eqref{eqn:pref-ex} corresponds to the preference vector
$$
\v_{\succeq} = 
\begin{bmatrix}
\entry{v}{11} \\
\entry{v}{10} \\
\entry{v}{01} \\
\entry{v}{00} \\
\end{bmatrix}
=
\begin{bmatrix}
3\\ 0 \\ 2 \\ 1
\end{bmatrix}.
$$
Conversely, any vector $\p \in \PP^n$ induces a preference ordering $\succeq_{\p}$ where we rank the outcomes  $x \succeq y$ whenever $\p(x) \geq \p(y)$.   For convenience, we define  $\character(\p) = \character(\succeq_{\p})$. 
Note that we have listed the entries of $\v_{\succeq}$  in reverse lexicographical order. This  aligns with two standard conventions for describing completely separable preferences: (a) the election outcome $11 \cdots 1$ is typically most preferred,  and  (b) the singleton outcomes satisfy
$$
100 \cdots 0 \succ 010 \cdots 0 \succ  \cdots \succ 000 \cdots 1.
$$
If not, we can remedy this situation by negating the failing questions to achieve (a)  and  then reordering the questions to achieve (b). 

Naturally, our preference construction hinges upon picking a useful basis for the preference space $\PP^n$. 
 We use hatted notation to denote the reverse bijection from $\ZZ_2^n$ to $\powerset{[n]}$:
\begin{equation}
\label{eqn:outset-def}
\outset{x} = \{ i \in [n] \mid x_i =1 \}.
\end{equation}
For example, $\outset{10110} = \{ 1, 3, 4 \}.$ We also define the parity indicator function on $\powerset{[n]}$
\begin{equation}
\label{eqn:even}
\even{S} = \left\{ \begin{array}{ll}
1 & \mbox{if } |S| \mbox{ is even,} \\
0 & \mbox{if } |S| \mbox{ is odd.} \\
\end{array}
\right.
\end{equation}

\begin{definition}
The \emph{voter basis} $\VV_n = \{\v_A \mid A \subset [n] \}$ is the collection of vectors whose entries $\v_A(x)$ are indexed by the outcomes $x \in \ZZ_2^n$ given by
\begin{equation}
\label{eqn:voterentry}
\entry{\setvec{A}}{x} 
= \even{\outset{x} \cap A}.
\end{equation}
\end{definition}

The voter basis  $\VV_3$ is shown in Table \ref{table:P3}. For each $A \in [n]$, the entries of $\v_A$ only take on two values: 0 and 1. Therefore, the preference ordering $\succeq_{\v_A}$ partitions $\powerset{[n]}$ into two equal parts: the preferred subsets and unpreferred subsets of $[n]$.
\begin{table}[ht]
$$
\begin{array}{ |c|c||c|c|c|c|c|c|c|c|c| } 
\hline
\mbox{subset} &   & \setvec{\{1,2,3\}} & \setvec{\{1,2\}} & \setvec{\{1,3\}} & \setvec{\{1\}} & \setvec{\{2,3\}} & \setvec{\{2\}} & \setvec{\{3\}} & \setvec{\emptyset}  \\ 
\hline
 &  \mbox{bitstring} & \setvec{111} & \setvec{110} & \setvec{101} & \setvec{100} & \setvec{011} & \setvec{010} & \setvec{001} & \setvec{000}  \\ 
  \hline
  \hline
\{1,2,3\}& 111 & 0 & 1 & 1 & 0 & 1 & 0 & 0 & 1 \\ 
\{1,2\}& 110 & 1 & 1 & 0 & 0 & 0 & 0 & 1 & 1 \\ 
\{1,3\}& 101 & 1 & 0 & 1 & 0 & 0 & 1 & 0 & 1 \\ 
\{1 \}& 100 & 0 & 0 & 0 & 0 & 1 & 1 & 1 & 1 \\ 
\{2,3 \}& 011 & 1 & 0 & 0 & 1 & 1 & 0 & 0 & 1 \\ 
\{2 \}& 010 & 0 & 0 & 1 & 1 & 0 & 0 & 1 & 1 \\ 
\{3\}& 001 & 0 & 1 & 0 & 1 & 0 & 1 & 0 & 1 \\
\emptyset & 000 & 1 & 1 & 1 & 1 & 1 & 1 & 1 & 1 \\
 \hline
\end{array}
$$
\caption{The voter basis for $\PP^3$.}
\label{table:P3}
\end{table}
Along with their simple structure, the voter basis vectors have nice separability properties.
\begin{theorem}
\label{thm:voter-basis}
The voter basis $\VV_n$ has the following properties:
\begin{enumerate}
\item[(a)]
$\VV_n$  is a basis for $\PP^n$. 
\item[(b)] 
The preference ordering $\succeq_{\v_A}$ induced by basis vector $\v_A$ is separable on $S \subset [n]$ if and only if $A \subseteq S$ or $A \cap S = \emptyset.$ Equivalently,
$$
\character(\v_A) = \{ S \mid A \subset S \mbox{ or } A \cap S = \emptyset \}.
$$ 
\end{enumerate}
\end{theorem}
While we do not take up the question of completely separable preferences in this current work, we are hopeful that the voter basis will be useful for illuminating this important (and difficult) family of linear orders. The voter basis also has deep connections to representation theory:  that we originally developed $\VV_n$ using representation theory for the hyperoctahedral group $\mathbb{Z}_2\wr S_n$. To maintain the focus of this exposition, we defer those connections to future work, and instead provide an elementary proof that $\VV_n$ is a basis for $\PP^n$. 

We conclude this section by drawing connections to previous research that employs vector representations in the study of election preferences. Hodge and Klima \cite{hodge2} represent a strict preference order of  as a column vector of bitstrings, with the voter's $i$th preference appearing in the $i$th row. Treating each row as a vector in $\ZZ_2^n$, we obtain a ${2^n \times n}$  \emph{binary preference matrix}. For example, the preference order of
equation \eqref{eqn:pref-ex} corresponds to the $4 \times 2$ binary preference matrix
$$
\begin{bmatrix}
1 & 1 \\ 1 & 0 \\ 0& 0 \\ 0 &1
\end{bmatrix}.
$$
This representation has proven quite useful in many of the constructions mentioned above. As a side note, the absence of an algebraic structure for these matrices was part of the motivation for our definition of the preference space $\PP^n$.  Looking at election outcomes more globally, Daughtery et al.~\cite{daugherty} introduced the \emph{profile space} $M^n \cong \QQ^{n!}$ to decompose an election according to the actual ballots cast. For example, a ballot for a ranked choice election with $n$ candidates corresponds to a permutation of $[n]$. Using a basis $\{ \v_{\sigma} \mid \sigma \in S_n \}$, where we view $\sigma \in S_n$ a linear ordering of $[n]$, the collection of voter ballots corresponds to the linear combination
$\sum_{\sigma \in S_n} a_{\sigma} \v_{\sigma}$ where $a_{\sigma}$ is the number of ballots cast with candidate ranking $\sigma$. 
To capture such aggregate behavior of the electorate in the preference space $\PP^n$, we would create a linear combination of the preference vectors across the electorate. Simplifying would give a single preference vector that captures the overall utility score for each election outcome. 
Finally, we note that 
a preference vector $\v \in \PP^n$ is equivalent to the \emph{value function} as defined in Bradley et al. \cite{bradley}, though our  vector space viewpoint is crucial to the methods herein.

\subsection{Roadmap}

The remainder of this paper is organized as follows. 
Section \ref{chap:preference} develops the voter basis. We   prove Theorem \ref{thm:voter-basis}, consider the separability properties of voter basis vectors, and introduce some helpful notation for  describing and combining partial outcomes. We then introduce a rank function on $\powerset{[n]}$ and use it to construct a nonseparable vector whose coefficients will be useful later on.
 In Section \ref{sec:trees}, we prove Theorem \ref{thm:tree-char}: tree characters are admissible.
 We conclude in Section \ref{chap:conclusion}, suggesting some directions for future research.

%
%

\section{The Voter Basis}
\label{chap:preference}

In this section,  we  formulate and prove some elementary results using bitstring notation,
and  prove Theorem \ref{thm:voter-basis},
Let $X_S$ denote the set of all bitstrings on $S \subset [n]$. Taking $S=[n]$, we define 
$\bitstrings{[n]} = \ZZ_2^n$ to be the set of all possible outcomes. Similarly, we define $\bitstrings{S}$ to be the set of all partial outcomes on the subset $S$. 
The simplest preference order on $S $ arises when  a voter is indifferent between all the outcomes.

\begin{definition}
A set $S$ is  {\bf trivially separable} with respect to $\succeq$ if for all $x_S, y_S\in \bitstrings{S}$ and all $u_{\comp{S}}\in \bitstrings{\comp{S}}$, we have
\begin{equation*}
    x_Su_{\comp{S}} \sim y_Su_{\comp{S}}.
\end{equation*}
\end{definition}

\begin{lemma}
\label{lemma:trivial-comp}
If $S$ is trivially separable then $\comp{S}$ is separable.
\end{lemma}

\begin{proof}
Consider  $x_S, y_S \in \bitstrings{S}$ and $u_{\comp{S}}, v_{\comp{S}} \in \bitstrings{\comp{S}}$. Suppose that $x_S u_{\comp{S}} \succeq x_S v_{\comp{S}}$. Then
$$y_S u_{\comp{S}} \sim  x_S u_{\comp{S}} \succeq x_S v_{\comp{S}} \sim  y_S v_{\comp{S}},$$
so that $y_S u_{\comp{S}} \succeq y_S v_{\comp{S}}$.
In other words, our preference on the outcomes on $\comp{S}$ is independent of the outcome on $S$.
\qed \end{proof}

Bradley, Hodge and Kilgour \cite{bradley} showed that set intersections preserve separability. We include a proof as an opportunity  to acquaint the reader with the general flow of the bitstring proofs that follow.

\begin{lemma}
\label{lemma:intersect}{\cite{bradley}}
If $S$ and $T$ are separable with respect to $\succeq$, then so is $S \cap T$.
\end{lemma}

\begin{proof}
We partition each bitstring $z$ as 
$$z = z_{S \cap T} \, z_{\comp{S \cap T}} = z_{S \cap T}  z_{S - T} \, z_{T - S} \, z_{\comp{S \cup T}}.$$
Suppose that $x_{S \cap T} u_{\comp{S \cap T}} \succeq y_{S \cap T} u_{\comp{S \cap T}}$, and let $v_{\comp{S \cap T}}$ be any other bitstring on $\comp{S \cap T}$.
Then
\begin{align*}
x_{S \cap T} u_{\comp{S \cap T}} &\succeq y_{S \cap T} u_{\comp{S \cap T}} \\
(x_{S \cap T} u_{S - T}) u_{T - S} u_{\comp{S \cup T}} &\succeq 
(y_{S \cap T} u_{S - T}) u_{T - S} u_{\comp{S \cup T}} \\
(x_{S \cap T} u_{S - T}) v_{T - S} v_{\comp{S \cup T}} &\succeq 
(y_{S \cap T} u_{S - T}) v_{T - S} v_{\comp{S \cup T}} &\mbox{since } S \mbox{ is separable} \\
(x_{S \cap T} v_{T - S} ) u_{S - T}  v_{\comp{S \cup T}} &\succeq 
(y_{S \cap T} v_{T - S}) u_{S - T}  v_{\comp{S \cup T}} \\
(x_{S \cap T} v_{T - S} ) v_{S - T}  v_{\comp{S \cup T}} &\succeq 
(y_{S \cap T} v_{T - S}) v_{S - T}  v_{\comp{S \cup T}} &\mbox{since } T \mbox{ is separable}  \\
x_{S \cap T} v_{\comp{S \cap T}} &\succeq y_{S \cap T} v_{\comp{S \cap T}}.
\end{align*}
Therefore $S \cap T$ is separable with respect to $\succeq$.
\qed \end{proof}


\label{sec:basis}

We now proving that $\VV_n$ is a basis for the preference space $\PP^n$, and then investigate the separability properties of a voter basis vector $\v_S$.
The following elementary lemma, suggested to us by Jeremy Martin  \cite{martin},  leads to a quick proof that $\VV_n$ is a basis. 

\begin{lemma}
\label{lem:martin}
Let $W_n$ be the $2^n \times 2^n$ matrix whose entries are indexed by subsets of $[n]$ and whose $(S,T)$-th entry is $$W_n(S,T) = (-1)^{| S \cap T|}.$$ Then $\det(W_1)=-2$ and  $\det(W_n) = 2^{n 2^{n-1}}$ for $n \geq 2$.
\end{lemma}

\begin{proof}
We have 
$$
W_1 = 
\left[
\begin{array}{cc}
1& 1 \\
1 & -1
\end{array}
\right]
\qquad
\mbox{and}
\qquad
W_2 = 
\left[
\begin{array}{cc|cc}
1& 1 & 1 & 1 \\
1 & -1 & 1 & -1 \\
\hline
1 & 1 & -1 & -1 \\
1 & -1 & -1 & 1
\end{array}
\right] = \left[
\begin{array}{c|c}
W_1 & W_1 \\
\hline
W_1 & -W_1 \\
\end{array}
\right],
$$
where we have used the  ordering $(\emptyset, \{1 \} )$ and $(\emptyset, \{1 \} ,  \{2 \}, \{1,2 \} )$ for the rows and columns of $W_1$ and $W_2$, respectively.
Clearly, $\det(W_1) = -2$. Elementary row operations on $W_2$ yield
$$
W_2 = \left[
\begin{array}{c|c}
W_1 & W_1 \\
\hline
W_1 & -W_1 \\
\end{array}
\right]
\sim
\left[
\begin{array}{c|c}
W_1 & W_1 \\
\hline
0 & -2W_1 \\
\end{array}
\right]
\sim
\left[
\begin{array}{c|c}
W_1 & 0 \\
\hline
0 & -2W_1 \\
\end{array}
\right]
$$
and these row addition operations do not change the determinant. Therefore $\det(W_2) = (-2)^2 \det(W_1)^2 = 2^4$. The same matrix structure holds for $W_n$ in terms of $W_{n-1}$, where we order by subsets of $[n-1]$ followed by subsets containing element $n$. Induction gives
$$
\det(W_n) = (-2)^{2^{n-1}} \det(W_{n-1})^2 = 2^{2^{n-1}} \left( 2^{(n-1)2^{n-2}} \right)^2
= 2^{n 2^{n-1}}.
$$
\qed \end{proof}

\begin{proof}[of Theorem \ref{thm:voter-basis}(a)]
Let $\w_S$ denote the column of $W_n$ indexed by $S \subset [n]$. Observe that 
 $\w_S = 2 \v_S - \v_{\emptyset}$ where $v_{\emptyset} = \one$ is the all-ones vector.
 Therefore $\VV_n$ is a basis for $\PP^n$.
\qed \end{proof}

It is important to note that the row/column order in this proof is different from the order displayed in Table \ref{table:P3}, which adheres to preference relation conventions. The recursive ordering is essential for the inductive proof. Also,
we could have used the $\w_S$ vectors as our basis, but the plentiful zeros of the $\v_S$ vectors  simplify the arguments below. 
Next, we introduce some terminology and notation for outcomes.

\begin{definition}
The  outcome $x$ is \emph{even (odd)} when  the size of the corresponding set $| \outset{x} |$ is even (odd).
The outcome $x$ is \emph{even in $A$ (odd in $A$)} when $|\outset{x} \cap A|$ is even (odd). 
Furthermore,  $x$ is even in $A$ if and only if the vector entry $\setvec{A}(x)=1$.
\end{definition}
For example, consider the set $\bitstrings{[3]} = \ZZ_3^2.$ The outcomes 011, 010, 001, and 000 are even in $\{1\}$ whereas the remaining four outcomes are odd in $\{1 \}$. The outcomes $000$, $001$, $110$ and $111$ are even in $\{1,2\}$, while the other four outcomes are odd in $\{1,2\}$.

Each voter basis vector $\setvec{A}$ induces a preference ordering  in which outcomes that are  even in $A$ are preferred to outcomes that are odd in $A$. The vector $\setvec{\emptyset}$ induces the trivial preference ordering (complete indifference). The vector $\setvec{[n]}$ prefers the even subsets of $[n]$ over the odd subsets. More generally, for nonempty $A \subset [n]$, the vector $\setvec{A}$ partitions the outcomes into a set of $2^{n-1}$ preferred outcomes and a set of $2^{n-1}$ undesirable outcomes. For example, when $n=3$, the voter basis vector $\setvec{\{1\}}=\setvec{100}$ induces the ordering
$$
\{ 011, 010, 001, 000 \} \succ \{111, 110, 101, 100 \}.
$$
When $n=4$, the vector $\setvec{\{1,2\}} = \setvec{1100}$ induces the ordering
\begin{align*}
& \{ 1111, 1110, 1101, 1100, 0011, 0010, 0001, 0000 \} \\
& \succ \{1011, 1010, 1001, 1000, 0111, 0110, 0101, 0100 \}.
\end{align*}

The following notation streamlines our nonseparability proofs.
Let $S$ be a set that we want to prove is nonseparable. 
We let $1_i$ and $0_i$ denote that the outcome of element $i \in S$ is fixed as 1 or 0, respectively on element $i \in S$. 
We let $0_*$ denote the partial outcome  that is all-zero on elements in $S$ that have not already been specified. 
For example, if $S= \{1,2,3,4,5\}$ then $x_S = 1_2 0_*$ denotes the outcome $01000$.
In the proofs below, we will often use this notation to construct sparse partial outcomes $x_S, y_S$ and $u_{\comp{S}}, v_{\comp{S}}$ so that $x_S u_{\comp{S}} \succ y_S u_{\comp{S}}$ while $x_S v_{\comp{S}} \prec y_S v_{\comp{S}}$.
As an example of the four resulting outcomes, suppose that $n=6$ and let $S=\{1,2,3\}$. 
Consider the partial outcomes $x_{S}=0_*$ and $y_{S}=1_20_*$ on $S$ and the partial outcomes
$u_{\comp{S}}=0_*$ and $v_{\comp{S}}=1_50_*$ on $\comp{S} = \{ 4,5,6\}$.
Concatenating each pairing gives 
$$
\begin{array}{rcl}
x_Su_{\comp{S}} &=& 000000, \\ 
y_Su_{\comp{S}} &=& 010000, \\ 
x_Sv_{\comp{S}} &=& 000010,\\ 
y_Sv_{\comp{S}} &=& 010010.\\ 
\end{array} 
$$

We are now ready to prove Theorem \ref{thm:voter-basis}(b): the vector $\setvec{A}$ induces a preference order that  is separable on $S$ if and only if $A\subseteq S$ or $A\cap S=\emptyset$.
Applying this theorem  for  $n=4$, the preference ordering induced by $\setvec{\{1\}}$ (or any other singleton set)  induces a completely separable ordering. The preference ordering induced by
$\setvec{\{ 1,2 \}}$ has character
$$
\{  \emptyset, \{3\}, \{4\}, \{1,2\}, \{3,4\}, \{1,2,3\}, \{1,2,4\}, \{1,2,3,4\} \}.
$$
Finally, $\setvec{\{1,2,3,4\}}$ induces a completely nonseparable ordering on $[4]$. 
As these examples show, the voter basis vectors have very useful separability properties.  Theorem \ref{thm:voter-basis} (b) reveals the potential of these basis vectors as building blocks for constructing preference orders. In particular, the nonseparable properties of $\setvec{[n]}$ will be essential for removing unwanted separabilities.

\begin{proof}[of Theorem \ref{thm:voter-basis}(b)]
Given $S \subset [n]$, let $x$ and $y$ be outcomes that are identical on $\comp{S}$. There are three cases to consider; we handle the two separable cases first.

\textbf{Case 1: $A\subseteq S$.} 
We decompose  $x=x_Su_{\comp{S}}=x_Ax_{S-A}u_{\comp{S}}$ and $y=y_Su_{\comp{S}}=y_Ay_{S-A}u_{\comp{S}}$.
We claim that preference relation between these outcomes  is independent of the shared binary digits  $u_{\comp{S}}$.
Indeed, if $x_A$ and $y_A$ are the same parity, then  both or neither are even in $A$, so that 
$\entry{\setvec{A}}{x_Su_{\comp{S}}} = \entry{\setvec{A}}{y_Su_{\comp{S}}}$
for all $u_{\comp{S}}\in X_{\comp{S}}$.
If  $x_A$ and $y_A$ are not the same parity, then we may assume that $x_A$ is even and $y_A$ is odd, so that
$\entry{\setvec{A}}{x_S u_{\comp{S}}} = 1 > 0 = \entry{\setvec{A}}{y_S u_{\comp{S}}}$
for all $u_{\comp{S}}\in X_{\comp{S}}$. Either way,  the preference between outcomes $x$ and $y$ depends only on the parities of $x$ and $y$ in $A$, which is independent of  $u_{\comp{S}}$. Therefore, $S$  is separable on $\setvec{A}$ whenever $A \subset S$.

\textbf{Case 2:  $A\cap S=\emptyset$.}
We decompose $x$ and $y$ as
$x=x_Su_Au_{\comp{S}-A}$ and $y=y_Su_Au_{\comp{S}-A}$.
The outcomes are identical on $A$, so their parity in $A$ is the same. Therefore
$\entry{\setvec{A}}{x_Su_Au_{\comp{S}-A}} =  \entry{\setvec{A}}{y_Su_Au_{\comp{S}-A}}$ 
for all $u_{\comp{S}}\in X_{\comp{S}}$ which means that $S$ is trivially separable on $\setvec{A}$.

\textbf{Case 3: $S \cap A \neq \emptyset$ and $A-S \neq \emptyset$.} Note that this includes the case where $\emptyset \subsetneq S \subsetneq  A$. 
We construct a pair of outcomes on $S$ that certify that $S$ is not separable.
Let $s \in S \cap A$ and let $a \in A - S$. Let $x_S = 0_*$ be the all-zero outcome and let $y_S = 1_s 0_*$ be the singleton outcome on $s$. Now let $u_{\comp{S}} = 0_*$ be the all-zero outcome and $v_{\comp{S}} = 1_a 0_*$ be the singleton outcome on $a$. We have
\begin{align*}
\v_A(x_S u_{\comp{S}}) &= 1 > 0 = \v_A(y_S u_{\comp{S}}) \\
\v_A(x_S v_{\comp{S}}) &= 0 < 1 = \v_A(y_S v_{\comp{S}}), 
\end{align*}
so that our preference between $x_S$ and $y_S$ depends on the outcome on $\comp{S}$. Therefore the set $S$ is not separable.
\qed \end{proof}

We conclude this subsection with a trio of results concerning the entries of $\setvec{T}$. The corrolaries will be used frequently in the next section to construct preference vectors with desired properties.

\begin{lemma}
\label{overlap}
Let $S, T \subset [n]$. Consider  outcomes $x = x_S u_{\comp{S}}$ and $y= y_S u_{\comp{S}}$ that agree on $\comp{S}$. We have  $\entry{\setvec{T}}{x}=\entry{\setvec{T}}{y}$ if and only if the partial outcomes $x_{S \cap T}$ and $y_{S \cap T}$ have the same parity.
\end{lemma}

\begin{proof}
The values
 of the entries $\setvec{T}(x)$ and $\setvec{T}(y)$  depend solely on the respective parity of the partial outcomes $x_{S \cap T}u_{T - S}$ and $y_{S \cap T}u_{T-S}$. These parities agree if and only if the parities of $x_{S\cap T}$ and $y_{S \cap T}$  agree. 
\qed \end{proof}

\begin{corollary}
\label{disjointed}
If $T \subset \comp{S}$ 
and 
the  outcomes $x = x_S u_{\comp{S}}$ and $y= y_S u_{\comp{S}}$  agree on $\comp{S}$, 
then $\setvec{T}(x)=\setvec{T}(y)$. \qed
\end{corollary}

\begin{corollary}
\label{supersets}
If $S \subset T$ 
and 
the  outcomes $x = x_S u_{\comp{S}}$ and $y= y_S u_{\comp{S}}$  agree on $\comp{S}$, 
then $\setvec{T}(x)=\setvec{T}(y)$ if and only if $x_S$ and $y_S$ have the same parity. \qed
\end{corollary}

Lemma \ref{overlap} highlights that fact that when outcomes agree on some subset $U$, then the voter basis vectors indexed by subsets of $U$ do  not contribute to preference differences between the two outcomes. 
The two corollaries are analogous to observations we made in the proof of Theorem \ref{thm:voter-basis}(b): the voter basis vector $\setvec{T}$ is trivially separable on sets disjoint from $T$ and separable on supersets of $T$.

\subsection{A Nonseparable Vector}

\label{sec:nonsep}

In this subsection, we introduce a rank function on $\powerset{[n]}$ and then use it to construct a preference vector $\w$ that induces a total order that is completely nonseparable.
 The proof is elementary,  but is valuable in that it  provides an opportunity to evaluate the separabilty of a preference vector, prior to grappling with the more intricate arguments below. Moreover, vectors derived from $\w$  will be used in the proof of Theorem \ref{thm:tree-char}.

\begin{definition} 
\label{def:rank}
The rank function $\rank: \powerset{[n]} \rightarrow [2^n]$ maps a subset $A \subset [n]$ to its position in the ordering of subsets of $[n]$  that lists the sets by increasing set size, and then within a fixed size, lists the sets lexicographically. 
\end{definition}

For example, when $n=3$, our  ordering is
\begin{equation}
\label{eqn:rank}
  \emptyset \prec \{1\} \prec \{2 \} \prec \{ 3 \} \prec \{1,2\} \prec \{1,3\} \prec \{2,3\} \prec \{1,2,3\},
\end{equation}
so that $\rho(\emptyset)=1$, and $\rho(\{1,3\}) = 6$, and so on.
Note that, in general, if  $A \subsetneq B \subset [n]$ then $\rank(A) < \rank(B)$. 
In other words,  the rank function $\rho$ provides a total ordering on $\powerset{[n]}$ that is monotone with respect to set containment.

 \begin{lemma}
 \label{thm:comp-nonsep}
The preference ordering induced by the vector
\begin{equation*}
\w = \sum_{ A \in \powerset{[n]}} 2^{\rank(A)} \setvec{A}
\end{equation*}
 is completely nonseparable. In other words,
$\character(\w) = \character(\succeq_{\w})= \{ \emptyset, [n] \}$.
 \end{lemma}

 \begin{proof} Let $A$ be a nontrival proper subset of $[n]$, and let $a \in A$ and $b \in [n] - A$.
We have
 $$
 \w(0_a0_*)  - \w(1_a0_*)  =
\sum_{S: a \in S \subseteq [n] }  2^{\rho(S)}  > 0.
 $$
 Meanwhile 
 $$
\w(0_a1_b0_*) - \w(1_a1_b 0_*0) = \sum_{S \in [n]-a-b} \left( 2^{\rho(S \cup \{ a \})} - 2^{\rho(S \cup \{ a,b \})} \right) < 0
 $$
 because the ranking function $\rho(S \cup \{ a \}) < \rho(S \cup \{ a,b\})$ for all $S \in  [n]-a-b$.
Our preference between the  outcomes $z_Au_{[n]-A}$ and $1_sz_{A-a}u_{[n]-A}$ depends on the value of $u_{[n]-A}$. Therefore the set $A$ is not separable.
 \qed \end{proof}
 
We now investigate the properties of $\w$, and of linear combinations of its components, making some observations that will be useful in the proof of Theorem \ref{thm:tree-char}.
First, we note that the entries of $\w$ are pairwise distinct.
 
 \begin{lemma}
 \label{lem:all-link-distinct}
 If $x \neq y$ are distinct outcomes on $[n]$, then $\w(x) \neq \w(y)$, so
$\w$ induces a total order on $\powerset{[n]}$.
 \end{lemma}
 
 \begin{proof}
 Let $a \in [n]$ be an element where $x$ and $y$ disagree. Without loss of generality,
 $x = 1_a x_{[n]-a}$ and $y=0_a y_{[n]-a}$.
The entry $\w(x)$ includes the summand $2^{\rank(\{ a\})}$ while $\w(y)$ does not. Since every positive integer has a unique binary representation, $\w(x) \neq \w(y)$.
 \qed \end{proof}
 
We close this subsection with one final observation. The sum of the coefficients of $\w$ is 
$$
 \sum_{A \in \powerset{[n]}} 2^{\rank(A)} = \sum_{k=1}^{2^n} 2^k = 2^{2^n+1} -2 < 2^{2^n+1}.
$$
Later on, in the proof of Theorem \ref{thm:tree-char},  we will break unwanted symmetries by adding a small vector whose nonzero coefficients are drawn from
$2^{-2^n-1} \w  $. We choose this scaling of $\w$ because its coefficient sum is strictly less than 1, which will be smaller than all the other nonzero coefficients.  Crucially, the impact of this small vector will be inconsequential, except when comparing entries that are otherwise equal.   
With that in mind, for every $A \in \powerset{[n]}$, we introduce the constant
\begin{equation}
\label{eq:coeflink}
\coeflink_A= 2^{\rank(A)-2^{2^n}-1} < 2^{-2^n}.
\end{equation}
We can now   generalize the previous lemma.

\begin{lemma}
 \label{lem:link-distinct}
Let $\emptyset \neq \calS \subset \powerset{[n]}$ be a family of subsets of $[n]$, and let $x \neq y$ be distinct outcomes on $[n]$ such that there is at least one $S \in \calS$ where $x_S$ and $y_S$ have different parities.
Let 
$$
\u =  \sum_{ A \in \calS } d_A \setvec{A}.
$$
Then $\u(x) \neq \u(y)$ and   $-1 < \u(x) - \u(y) < 1$.
\end{lemma}

 \begin{proof}
 Without loss of generality,
 $x_S$ is odd and $y_S$ is even.
The entry $\u(x)$ includes the summand $d_S$ while $\u(y)$ does not. As in the proof of
Lemma \ref{lem:all-link-distinct},  $\u(x) \neq \u(y)$, since the summands are scaled powers of 2.
We have $-1 < \u(x) - \u(y) < 1$ since $\sum_{A \in \powerset{[n]}} d_A < 1$.
 \qed \end{proof}


\section{Tree Characters}
\label{sec:trees}

We now turn to the proof of Theorem \ref{thm:tree-char}, In Section \ref{sec:tree-thm}, we develop some terminology and then give an expanded version of Theorem \ref{thm:tree-char}.  
This detailed formulation specifies the linear combination of $\VV_n$  that produces a preference order with the desired tree character.  We give then two examples of preference vectors constructed using the tree character theorem.   In Section \ref{sec:tree-proof}, we  prove  Theorem \ref{thm:tree-char}, deferring  the details of four technical lemmas to Section \ref{sec:tree-lemma}.

\subsection{The Tree Character Theorem}
\label{sec:tree-thm}

We start by building on Definition \ref{def:tree-char}, where we introduced the tree character. 
We then formulate a more explicit version of Theorem \ref{thm:tree-char} which gives a precise  formula for the preference vector that induces the desired tree character.

\begin{definition}
\label{def:tree-char2}
Let $\ch$ be a tree character and let $A , B \in \ch$. If $B$ covers $A$, then $B$ is the \emph{parent} of $A$, and $A$ is the \emph{child} of $B$. The children of $B$ are called \emph{siblings}. The \emph{$k$th generation} of sets consists of all sets that are at distance $k$ from the root $[n]$. For $A \neq \emptyset$, we use $g(A)$ to denote the generation of $A$. \emph{Ancestors} and \emph{descendants} 
are defined in the natural way.
\end{definition}

In a tree character, every proper nonempty set has a unique set that covers it. For example, the collection of sets
$$
\ch_1 = \{\emptyset,\{1\}, \{2\},\{3\},\{1,2\}, \{4,5\},\{3,4,5\},\{1,2,3,4,5,6,7,8\}\}
$$
is a tree character. Figure \ref{fig:tree-char-example}(a) represents $\ch - \{ \emptyset \}$ as a rooted tree.
We have $g(\{1,2\})=1$ and $g(\{4,5\})=2$ and $g(\{4\})=3$.
On the other hand, the character
$$
\ch_2 = \{ \emptyset, \{1\}, \{1,2\}, \{1,3\}, \{1,2,3\} \}
$$
is not a tree character, since both $\{1,2\}$ and $\{1,3\}$ cover $\{1\}$.

\begin{figure}[ht]
\begin{center}

\begin{tabular}{ccc}

\begin{tikzpicture}[every node/.style={draw,rectangle}]
\tikzset{sibling distance=20pt}
\tikzset{edge from parent/.style=
{draw,
edge from parent path={(\tikzparentnode.south)
-- +(0,-8pt)
-| (\tikzchildnode)}}}

\Tree [.{1 2 3 4 5 6 7 8 } 
        [.{1 2} [.1 ] [.2 ] ]
        [.{3 4 5} [.{3} ]
               [.{4 5} [.4 ]  ] ]
    ]

\end{tikzpicture}

&
\qquad \qquad
&

\begin{tikzpicture}[every node/.style={draw,rectangle}]
\tikzset{sibling distance=20pt}
\tikzset{edge from parent/.style=
{draw,
edge from parent path={(\tikzparentnode.south)
-- +(0,-8pt)
-| (\tikzchildnode)}}}

\begin{scope}[shift={(8,0)}]

\Tree [.{1 2 3 4 5 6 7 8} 
        [.{1 2} [.1 ] [.2 ] ]
        [.{3 4 5} [.{3} ]
               [.{4 5} [.4 ] [.\node[dashed, fill=gray!20]{5};  ] ] ]
        [.\node[dashed, fill=gray!20]{6 7 8};  ]               
    ]
        
\end{scope}

\end{tikzpicture}

\\
\\

(a) && (b)

\end{tabular}

\end{center}
\caption{(a) The Hasse diagram (omitting the set $\emptyset$) of a tree character $\ch_1$
(b) The haunted Hasse diagram of $\ch_1$. The   ghost children are shown with shaded background and dashed outlines.}
\label{fig:tree-char-example}
\end{figure}
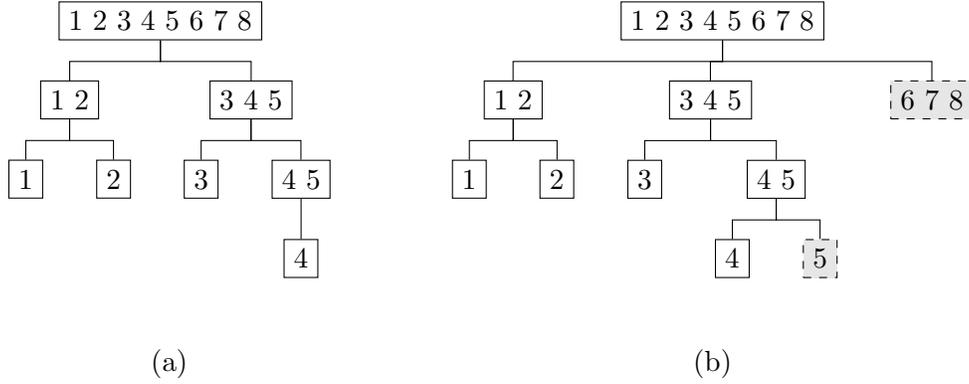

As we construct our preference vector for a tree character $\ch$, we will also need to keep track of the elements of $[n]$ that appear in generation $k$, but do not appear in generation $k+1$.  For example, in $\ch_1$, the set $\{ 4, 5\}$ has one child $\{ 4\}$ but the element 5 is ``missing'' from the next generation. For convenience, we collect these missing elements into sets of \emph{ghost children}.
\begin{definition} Let $\ch$ be a tree character and consider
 $A \in \ch$ with children $A_{1}, A_{2}, \ldots, A_{k}$ where $k \geq 1$. If $\cup_i A_{i} \neq A$, then the \emph{ghost child} of $A$ is
$A - \cup_i A_{i} $. 
The Hasse diagram that includes the ghost children is called a \emph{haunted Hasse diagram}.
\end{definition}
Note that if $A$ does not have any children, then it does not have a ghost child either. Figure \ref{fig:tree-char-example}(b) shows the haunted Hasse diagram for the tree character $\ch_1$.
With the addition of ghost children,  every element in the set $[n]$ appears in exactly one leaf of the  haunted Hasse diagram.

During our construction, we will use  ghost children to prevent (unwanted) unions of siblings from becoming  separable. For example, in Figure \ref{fig:tree-char-example}(b), the children of the set $[8]$ are $\{1,2\}$ and $\{3,4,5\}$. We will use the ghost child $\{6,7,8\}$ to prevent  the set $\{1,2,3,4,5\}$ from also being separable. More precisely, to break unwanted separability on unions of siblings $ A_i \cup A_j$, we include a tiny vector in the direction of $\v_{A_i \cup A_j}$, as described in Theorem \ref{thm:tree-char} below.

\begin{definition} 
\label{def:siblink}
Let $\ch$ be a tree character.
Let $A \in \ch$ with children $A_1,A_2,\ldots,A_k$, where one of these sets might be a ghost child of $A$. Then the \emph{sibling linkage} $\linkage(A)$ of these children is
$$\linkage(A) = \{ A_i \cup A_j : 1 \leq i < j \leq k \}$$
Elements of the sibling linkage $\linkage(A)$ are called \emph{siblinks}. The set
$$
\linkage = \bigcup_{A \in \ch} \linkage(A)
$$
is the \emph{sibling linkage} of the tree character $\ch$.
\end{definition}

We can now state a more detailed version of Theorem \ref{thm:tree-char}

 \begin{nonum}[Theorem~\ref{thm:tree-char}]
 Every tree character is admissible. More precisely,
consider a tree character
$\ch$ on $[n]$. Let $\base=2 - 2^{-(n-1)}$. For $\emptyset \neq A \in \ch$, let  $\coef_A=\base^{g(A)}$ and let
 $\coeflink_B = 2^{\rank(B)-2^{2^n}-1}$ where $\rank(B)$ is the rank of set $B$.
Define
\begin{equation}
\label{eqn:vchi}
\setvec{\ch}= \sum_{A \in \ch} \coef_A \setvec{A} + \sum_{B \in \linkage} \coeflink_B \setvec{B}.
\end{equation}
Then $\ch$ is the collection of separable sets in the ordering induced by the preference vector $\setvec{\ch}$. In other words, $\character(\setvec{\ch})=\ch$.
\end{nonum}

We illuminate the form and function of the coefficients  $c_A$ and $d_B$ in Section \ref{sec:tree-proof} below. For now, it is enough to mention that $c_A$ creates the separability of $A \in \ch$, and $d_B$ breaks unwanted separabilities  of some sets outside of $\ch$. Finally, we note that $\base$ is essentially equal to 2, but choosing $\base=2$ would invalidate Lemma \ref{lem:notsiblings} below.

We conclude this section with two examples which use Theorem \ref{thm:tree-char} to construct  preference orderings with tree characters. First, consider the simple example for $n=3$ with tree character
$$
 \ch_3=\{ \emptyset, \{1\}, \{2\}, \{1,2,3\} \}
$$
The root $\{1,2,3\}$ has children $\{1\}, \{2\}$ and ghost child $\{3\}$. Therefore, the siblinks are
$$
\linkage_3 = \{ \{1,2\}, \{1,3\}, \{2,3\} \}.
$$
We have $\alpha = 2 - 2^2 =7/4$ and the siblink coefficients are of the form $2^{\rho(B) - 9}$
for the rank function $\rho$ corresponding to  equation \eqref{eqn:rank}. Our desired vector is
$$
\setvec{\ch_3}
=
\setvec{\{1,2,3\}} + \frac{7}{4} \setvec{\{1\}} + \frac{7}{4} \setvec{\{2\}}   + \frac{1}{16} \setvec{\{1,2\}}  + \frac{1}{8} \setvec{\{1\}} + \frac{1}{4} \setvec{\{2,3\}}.
$$
Using the voter basis vectors listed in Table \ref{table:P3}, the entries of $16 \, \setvec{\ch_3}$ are 
$$
 \begin{array}{c|c|c|c|c|c|c|c}
 111 & 110 & 101 & 100 & 011 & 010 & 001 & 000 \\
 \hline
7 & 17 & 46 & 32 & 48 & 30 & 57 & 79 
\end{array} 
$$
which corresponds to preference order
$$
111 \prec 110  \prec 010 \prec 100 \prec 101 \prec 011 \prec 001 \prec 000.
$$
It is easy to check that the separable sets are precisely those in $\ch_3$.

Finally, we construct a preference vector in $\PP^9 \cong \QQ^{2^9}$ for the 
 tree character 
 \begin{align}
  \nonumber
 \ch_4&=\{
 \emptyset,
 \{1\},  \{2\} ,\{7\},\{8\},
 \{3,4\},  \{5,6\}, \{7,8,9\}, \\
&\qquad  \{1,2,3,4\},  \{5,6,7,8,9\},
  \{1,2,3,4,5,6,7,8,9\}\}.
  \label{eqn:ch4}
 \end{align}
 The haunted Hasse diagram of $\ch_4$ is shown in Figure \ref{fig:char-to-vec-example}.

\begin{figure}[ht]
\begin{center}
\begin{tikzpicture}[every node/.style={draw,rectangle}]
\tikzset{sibling distance=20pt}
\tikzset{edge from parent/.style=
{draw,
edge from parent path={(\tikzparentnode.south)
-- +(0,-8pt)
-| (\tikzchildnode)}}}

\Tree [.{1 2 3 4 5 6 7 8 9} 
            [.{1 2 3 4} 
                [.1 ]
                [.2 ]  
                [.\node[dashed, fill=gray!20]{3 4};]]
            [.{5 6 7 8 9} 
                [.{5 6} ] [.{7 8 9} 
                    [.7 ] [.8 ] [.\node[dashed, fill=gray!20]{9};] ] ]
    ]

\end{tikzpicture}

\end{center}
\caption{The haunted Hasse diagram for the tree character $\ch_4$.}
\label{fig:char-to-vec-example}
\end{figure}
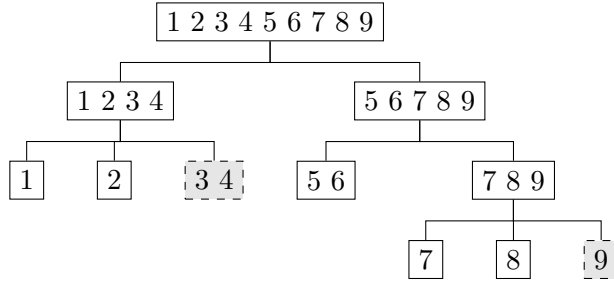

There are four nonempty sibling linkages for $\ch_4$:
\begin{align*}
    \linkage([9])&= \{ \{1,2,3,4,5,6,7,8,9\} \}, \\
    \linkage(\{1,2,3,4\})&=\{ \{1,2\}, \{1,3,4\}, \{2,3,4\} \}, \\
    \linkage(\{5,6,7,8,9\})&= \{ \{5,6,7,8,9\} \}, \\
    \linkage(\{7,8,9\})&= \{\{7,8\},\{7,9\},\{8,9\}\}.
\end{align*}
So the set of all siblinks is 
\begin{align*}
\linkage_4 =& \{
\{1,2\}, \{7,8\},  \{7,9\}, \{8,9\},
\{1,3,4\},  \{2,3,4\} ,   \\
& \qquad \{5,6,7,8,9\} , \{1,2,3,4,5,6,7,8,9\} 
\}.
\end{align*}

As described in equation \eqref{eqn:vchi}, the preference vector $\setvec{\ch_4}$ consists of two summations. First, 
 we create a linear combination of basis vectors indexed by elements of $\ch_4$. The coefficient of $\setvec{A}$ is determined by the generation of $A$ in the haunted Hasse diagram. This gives us the first part of our vector $\setvec{\ch_4}$:
\begin{align}
\nonumber
& 1\setvec{[9]}+\base \setvec{\{1,2,3,4\}}+ \base \setvec{\{5,6,7,8,9\}}+\base^2\setvec{\{1\}}+\base^2\setvec{\{2\}}+\base^2\setvec{\{5,6\}} \\
\label{eqn:vec1}
& \quad +\base^2\setvec{\{7,8,9\}} +\base^3\setvec{\{7\}}+ \base^3\setvec{\{8\}},
\end{align}
where $\base=2-2^{-8}$.
Next, we create a linear combination of basis vectors indexed by the siblinks 
\begin{align}
\label{eqn:vec2}
\frac{1}{2^{2^{9}+1}} \sum_{B \in \linkage_4} 2^{\rho(B)} \setvec{B}
\end{align}
where $\rho(B)$ is the rank of siblink $B \in \linkage_4$ in the ordering of $\powerset{[4]}$.
We obtain $\setvec{\ch_4}  \in \PP^{9}$ by adding expressions  \eqref{eqn:vec1} and \eqref{eqn:vec2}.
We can routinely check that $\character(\setvec{\ch_4})=\ch_4$, though this is best done via mathematical software.

\subsection{Proof of the Tree Character Theorem}
\label{sec:tree-proof}

In this subsection, we prove Theorem \ref{thm:tree-char} via four technical lemmas. The proofs of those lemmas  are deferred to the next subsection. 
We start with  a few observations about  the coefficients in Theorem \ref{thm:tree-char}. First, if $A \in \ch \cap \linkage$ then the coefficient of $\v_A$ is $\coef_A + \coeflink_A.$
Second, we have $\coef_A \geq 1$ for every $A \in \ch$, while $\sum_{A \in \linkage} {d_A} < 1$.
A third  property is described by the following lemma, which illuminates the choice of the constant $\base=2-2^{-(n-1)}$.
\begin{lemma}
\label{lem:basesum}
Suppose that $n \geq 3$. Let $\base= 2-2^{-(n-1)}$. For $1 \leq  m \leq n-1$, we have
$$
0 < \base^m - \sum_{i=0}^{m-1} \base^i = \base^m - \frac{\base^{m} - 1}{\base-1}< 1,
$$
and consequently, for  $1 \leq r < s \leq n-1$, we have
$$
 \sum_{i=r}^{s-1} \base^i  < \base^s < \base^r +  \sum_{i=r}^{s-1} \base^i.
$$

\end{lemma}

While we are tempted to take $\base=2$, this would replace the final inequality with  the equality $2^s = 2^r +  \sum_{i=r}^{s-1} 2^i$, which breaks the proof of Lemma \ref{lem:notsiblings} below. The choice of $\base=2-2^{-(n-1)}$ gives the behavior we need for tree characters of $[n]$.

\begin{proof}
The first inequality chain is equivalent to
$ -1 < \base^m (\base-2) < \base-2$. Dividing by the negative quantity $\base-2$ gives
$$
1 < \base^m < \frac{1}{2-\base},
$$
which clearly holds for $\base=2-2^{-(n-1)}$ and $1 \leq m \leq n -1$.
The second  inequality chain follows directly from the first.
\qed \end{proof}

The proof of Theorem \ref{thm:tree-char} is quite technical. We  defer the details to a series of four lemmas. Consider the preference ordering induced by $\setvec{\ch}$ of equation \eqref{eqn:vchi}. We must show that it is separable on sets in $\ch$ and nonseparable on all other sets.

\begin{lemma}
\label{lem:tree-charsep}
 The preference vector $v_{\ch}$ is separable on every element in $\ch$.
\end{lemma}

We prove this lemma in the next subsection.
Turning to sets that are not members of $\ch$, we introduce some additional definitions to partition these sets into three categories.

 \begin{definition}
 \label{def:unbreakable}
 Let $\ch$ be a tree character.
 A nonempty set  $B \in  \ch$ is {\bf unbreakable} when $B$ cannot be written as the union
 of sets  in $\ch - \{ B \}.$ The set of all unbreakable sets is denoted $U(\ch)$.
 \end{definition}
 
 The collection of unbreakable sets is a superset of the  join-irreducibles of the lattice $\ch$ (which are elements that cover exactly one other element). For example, in the tree character
 $\ch_4$ from equation \eqref{eqn:ch4} and Figure \ref{fig:char-to-vec-example}, the join-irreducibles are
 $$
 \{1\}, \{2\}, \{5,6\},  \{7\}, \{8\}
 $$
while $U(\ch_4)$ also includes the sets
$$
\{1,2,3,4\}, \{7,8,9\}.
$$

  \begin{definition}
   \label{def:construct}
 Let $\ch$ be a tree character and let $B \subset [n]$.
 Let  $$\Upsilon'_{\ch}(B) = \{A \in U(\ch) \mid A \subset B\}$$ 
 be the collection of $\ch$-unbreakable sets contained in $B$
 and let
 $$
 \Upsilon_{\ch}(B) = \{ A \in \Upsilon'_{\ch}(B) \mid \forall A' \in \Upsilon'_{\ch}(B), A \not\subset A' \}.
 $$ 
 be the collection of  maximal sets in  $\Upsilon'_{\ch}(B)$. 
The {\bf construct} of $B$ is
 $$
 \construct{B}{\ch} = \bigcup_{A \in \Upsilon_{\ch}(B)} A.
 $$
When $\construct{B}{\ch} = B$, we say that  $B$ is {\bf constructible}.  
\end{definition}

We make some elementary observations about constructs and unbreakable sets. First, if $A$ is unbreakable, then $A - \construct{A}{\ch} \neq \emptyset$. Second, for any $B \subset [n]$,
the tree structure of $\ch$ ensures that the collection $\Upsilon_{\ch}(B)$ of maximal $\ch$-unbreakable subsets of $B$ is pairwise disjoint. 
Third,  every set $A \in \ch$ is $\ch$-constructable: either $A$ is unbreakable, or it can be written as the disjoint union of  unbreakable sets. And finally, if $B$ can be written as the union of sets in the tree character $\ch$, then $B$ is constructible using maximal $\ch$-unbreakable subsets of $B$.

For an example, let us return to tree character $\ch_4$. 
The following four sets are not $\ch_4$-constructible:
$$
\{5\},\
\{1,2,4\},\
\{2,6\},\
\{1,2,7,9\},
$$
while these five sets are $\ch_4$-constructible:
\begin{align*}
\{1,2\}&=\{1\}\cup\{2\}, \\
\{2,5,6,8\}&=\{2\}\cup\{5,6\}\cup\{8\}, \\
\{1,2,3,4,7\}&=\{1,2,3,4\}\cup\{7\}, \\
\{1,2,7,8\}&=\{1\}\cup\{2\} \cup \{7 \} \cup \{8\}, \\
\{ 5,6,7,8,9 \}& = \{5,6 \} \cup \{7,8,9\}.
\end{align*}
Note that $\{1,2\}$ and $\{5,6,7,8,9\}$ are $\ch_4$-constructed from siblings in $\ch_4$. Meanwhile,  the other three sets are $\ch_4$-constructed using elements that are not siblings.

Here are the lemmas that handle sets that are not contained in $\ch$. Their proofs are deferred to the subsection that follows.

\begin{lemma}
\label{lem:notcon}
Consider a set $B \notin \ch$ that is not $\ch$-constructible. Then the set $B$ is not separable on $\setvec{\ch}$.
\end{lemma}

 \begin{lemma}
\label{lem:siblings}
Consider a constructible set $B\notin\ch$ where the elements of $\Upsilon_{\ch}(B)$ are siblings in $\ch$. The set $B$ is not separable on $\setvec{\ch}$.
\end{lemma}

 \begin{lemma}
\label{lem:notsiblings}
Consider a constructible set $B\notin\ch$ where at least two elements of $\Upsilon_{\ch}(B)$ are not siblings in $\ch$.  The set $B$ is not separable on $\setvec{\ch}$.
\end{lemma}


\begin{proof}[of Theorem \ref{thm:tree-char}]
If $B \in \ch$, then $B$ is separable by Lemma \ref{lem:tree-charsep}. Consider $B \notin \ch$. If $B$ is not $\ch$-constructible, then $B$ is nonseparable by Lemma \ref{lem:notcon}. If $B$ is constructible, then either this construction uses a set of siblings or does not. Lemmas \ref{lem:siblings} and \ref{lem:notsiblings} show that $B$ is not separable in either case.
In summary,  only elements of $\ch$ are separable. Therefore, $\character(\setvec{\ch})=\ch$.
\qed \end{proof}

All that remains is to justify these four  lemmas. The proofs become more intricate as we progress. In particular,  Lemma \ref{lem:siblings} requires ghost children and siblinks  to force non-separabilty among unions of siblings in the tree character. Let us begin.

\subsection{Proofs of the Tree Character Lemmas}
\label{sec:tree-lemma}

We start with a few elementary observations and some helpful notation. 
Let $\ch$
be a tree character of $[n]$ and let
$\linkage = \cup_{A \in \ch} \linkage(A)$
be the union of all siblinks of $\ch$. The vector $\setvec{\ch}$ constructed in Theorem \ref{thm:tree-char}
is $$
\setvec{\ch}= \sum_{A \in \ch} \coef_A \setvec{A} + \sum_{A \in \linkage} \coeflink_A \setvec{A}. 
$$
The coefficients are given by $\coef_A = \base^{g(A)}$ where $\base=2-2^{-(n-1)}$ and $g(A)$ is the generation of $A \in \ch$, and  $\coeflink_A = 2^{\rank(A)-2^n-1}$ where $\rank(A)$ is the rank of the set $A$ in our ordering of $\powerset{[n]}$ from Definition \ref{def:rank}. For convenience, we define $\coef_S = 0$ when $S \notin \ch$ and $\coeflink_S = 0$ when $S \notin \linkage$, so that we denote the coefficient of $S$ as $\coef_S + \coeflink_S$, regardless of whether $S$ is a member of the character or the sibling linkage.

Suppose that $A_{r} \in \ch$ is in the $r$th generation and that  $A_{r} \subset A_{r-1} \subset \cdots \subset A_{0} = [n]$ is its complete ancestral chain of supersets in $\ch$. Then Lemma \ref{lem:basesum} ensures that
\begin{equation}
\label{gparents}
\coef_{A_r} = \base^r > \sum_{j=0}^{r-1} \base^j = \sum_{j=0}^{r-1} \coef_{A_j}.
\end{equation}
Next, we observe that $| \linkage| < 2^n$ because each siblink is the union of two disjoint nonempty sets. Equation \eqref{eq:coeflink} yields
\begin{equation}
\label{eq:1sib}
 \sum_{A \in \linkage} \coeflink_A  \leq  \frac{| \linkage |}{2^n} < 1.
\end{equation}

We employ the following notation for partial sums of the coefficients of $\setvec{\ch}$ that are even with respect to a given set $A$. For  $\calS \subset  \ch$ and $\calT \subset \linkage$, we define
\begin{align*}
\coefsum{A}{\calS} &= \sum_{S \in \calS} \even{A \cap S}  \, \coef_S, \\
\coefsumlink{A}{\calT} &=  \sum_{T \in \calT} \even{A \cap T} \,  \coeflink_T, 
\end{align*}
where the parity indicator function $\even{\cdot}$ is defined in equation \eqref{eqn:even}.
It will  be convenient to use this same notation with a partial outcome $x_A$, in which case we define
$\coefsum{x_A}{\calS}= \coefsum{\outset{x_A}}{\calS}$ and
$\coefsumlink{x_A}{\calT}= \coefsumlink{\outset{x_A}}{\calT}$,
where we use the hatted notation of equation \eqref{eqn:outset-def}.
This partial sum notation gives a compact expression for the entry of $\setvec{\ch}$ corresponding to set $A$. Indeed, equation \eqref{eqn:voterentry} states that the  entry  $\entry{\setvec{B}}{A}=1$ if and only if $A$ is even in $B$. Therefore, the entry of $\setvec{\ch}$ corresponding to set $A$ is 
$$
\entry{\setvec{\ch}}{A} = \coefsum{A}{\ch} + \coefsumlink{A}{\linkage}.
$$
Given two sets $A$ and $B$, we will often need to compare $\entry{\setvec{\ch}}{A}$ with $\entry{\setvec{\ch}}{B}$. Equation \eqref{eq:1sib} leads to the handy observation
\begin{equation}
\label{eq:within1}
 \coefsum{A}{\ch} - \coefsum{B}{\ch} - 1  < \entry{\setvec{\ch}}{A} - \entry{\setvec{\ch}}{B}  <  \coefsum{A}{\ch} - \coefsum{B}{\ch} +1.
\end{equation}
In particular, $\coefsum{A}{\ch}-\coefsum{B}{\ch} \in \ZZ$, so equation \eqref{eq:1sib} allows us to ignore the fractional contribution from the siblink coefficients  when this difference is nonzero.

For a given set $B \subset [n]$, and a collection of subsets $\calT \subset \powerset{[n]}$, we will also be interested in members of $\calT$ that are  subsets of $B$, supersets of $B$ or disjoint from $B$.  
We define
\begin{align*}
\strictsubsets{B}{\calT} &= \{ T \in \calT: T \subsetneq B\}, \\
\subsets{B}{\calT} &=\strictsubsets{B}{\calT} \cup \{B \}, \\
\supersets{B}{\calT} &= \{ T \in \calT : B \subsetneq T \}, \\
\disjointsets{B}{\calT} &= \{ T \in \calT :  B \cap T = \emptyset \}.
\end{align*}
For $A \in \ch$, the tree structure of $\ch$ leads to the partitions
$
\ch = \subsets{A}{\ch} \cup \supersets{A}{\ch} \cup \disjointsets{A}{\ch}
$
and
$
\linkage = \subsets{A}{\linkage} \cup \supersets{A}{\linkage} \cup \disjointsets{A}{\linkage}.
$
Consequently, we   decompose   $\entry{\setvec{\ch}}{A}$ as
\begin{align*}
\entry{\setvec{\ch}}{A} 
&= \coefsum{A}{\subsets{A}{\ch}} + \coefsum{A}{\supersets{A}{\ch}}+  \coefsum{A}{\disjointsets{A}{\ch}} \\
& \qquad +\coefsumlink{A}{\subsets{A}{\linkage}} + \coefsumlink{A}{\supersets{A}{\linkage}}+  \coefsumlink{A}{\disjointsets{A}{\linkage}},
\end{align*}
 in the separability proofs that follow. But first, we prove a quick but useful lemma.
\begin{lemma}
\label{childdiff}
Let $A \in \ch$. 
Consider distinct outcomes $x =x_{A}u_{[n]-A}$  
and $y =y_{A}u_{[n]-A}$  that are identical  on $[n]-A$.
 If $\coefsum{x_{A}}{\subsets{A}{\ch}} \neq \coefsum{y_{A}}{\subsets{A}{\ch}}$ then
$$| \coefsum{x_{A}}{\subsets{A}{\ch}} - \coefsum{y_{A}}{\subsets{A}{\ch}}| \geq \coef_A \geq 1.$$
\end{lemma}

\begin{proof}
If $\coef_B$ is a summand in either $\coefsum{x_{A}}{\subsets{A}{\ch}}$ or $\coefsum{y_{A}}{\subsets{A}{\ch}}$, then  $B \subset A$, so that $\coef_B = 2^{g(B)} \geq 2^{g(A)} = \coef_A$. Since $\coef_A$ divides every term in both $\coefsum{x_{A}}{\subsets{A}{\ch}}$ and $\coefsum{y_{A}}{\subsets{A}{\ch}}$, it also divides their difference.
\qed \end{proof}

We are now prepared to prove our four lemmas. The non-separability proofs use the notation $x_{\{1,2,3\}} = 1_3 0_* = 001$ introduced in Section \ref{sec:basis} for constructing sparse outcomes on $[n]$.

%
%
\subsubsection{Proof of Lemma \ref{lem:tree-charsep}}

We prove that if $A \in \ch$ then the ordering induced by $\setvec{{\ch}}$ is separable on $A$.
Consider  distinct  partial outcomes $x_{A}$ and
$y_{A}$. Let $u_{[n]-A}$ be any partial outcome on $[n]-A$ and define
 $x=x_{A}u_{[n]-A}$ and
$y=y_{A}u_{[n]-A}$. We claim that the sign of the difference
$$
\entry{\setvec{\ch}}{x} - \entry{\setvec{\ch}}{y}
= \left( \coefsum{x}{\ch} + \coefsumlink{x}{\linkage}  \right) -
\left( \coefsum{y}{\ch} + \coefsumlink{y}{\linkage}  \right)
$$ 
(1) only depends upon $x_{A}$ and $y_{A}$, and (2)  is independent of the particular choice of $u_{[n]-A}$.
Recall that in the ordering induced by $\setvec{\ch}$, the inequality $\entry{\setvec{\ch}}{x} > \entry{\setvec{\ch}}{y}$ corresponds to the preference $x \succ y$. So this claim is equivalent to the separability of set $A$.

\textbf{Case 1:} $\coefsum{x}{\subsets{A}{\ch}} > \coefsum{y}{\subsets{A}{\ch}}$.   We claim that $x_{A}u_{\comp{A}} \succ y_{A}u_{\comp{A}}$. We have
$$
\entry{\setvec{\ch}}{x} - \entry{\setvec{\ch}}{y}
> \coefsum{x}{\ch}  -
\coefsum{y}{\ch} -1
$$
by \eqref{eq:within1}.
Corollary \ref{disjointed} shows that when $B \in [n] - A$, we have  $\entry{\setvec{B}}{x_{A}u_{[n]-A}}=\entry{\setvec{B}}{y_{A}u_{[n]-A}}$. 
Therefore, we can ignore the sets in $\disjointsets{A}{\ch} $ when calculating $\entry{\setvec{\ch}}{x} - \entry{\setvec{\ch}}{y}$.
We have
\begin{align*}
\entry{\setvec{\ch}}{x} - \entry{\setvec{\ch}}{y} 
&> \coefsum{x_{A}}{\subsets{A}{\ch}} + \coefsum{x_{A}}{\supersets{A}{\ch}} \\
& \qquad - \coefsum{y_{A}}{\subsets{A}{\ch}} - \coefsum{y_{A}}{\supersets{A}{\ch}} -1
\\
&\geq 
2^{g(A)} - \coefsum{y_{A}}{\supersets{A}{\ch}} -1 \\
& \geq 2^{g(A)} - (2^{g(A)} -1)  -1= 0
\end{align*}
by Lemma \ref{childdiff} and equation \eqref{gparents}. 
We conclude that $x_{A_l}u_{[n]-A_l} \succ y_{A_l}u_{[n]-A_l}$ for every choice of $u_{[n]-A_l}$.

\textbf{Case 2:}  $\coefsum{x}{\subsets{A}{\ch}} < \coefsum{y}{\subsets{A}{\ch}}$. We have $x_{A}u_{[n]-A} \prec
y_{A}u_{[n]-A}$  for every choice of $u_{[n]-A}$ by reversing the roles of $x$ and $y$ in Case 1.

\textbf{Case 3:}  $\coefsum{x}{\subsets{A}{\ch}} = \coefsum{y}{\subsets{A}{\ch}}$. We claim that 
\begin{equation}
\label{eq:case3}
\entry{\setvec{\ch}}{x} - \entry{\setvec{\ch}}{y} = 
\coefsumlink{x}{\linkage} - \coefsumlink{y}{\linkage}.
\end{equation}
As in Case 1, we can ignore the sets in $\disjointsets{A}{\ch}$.
Next, observe that $x_{A}$ and $y_{A}$ have the same parity. If this were not true, then exactly one of $\coefsum{x}{\subsets{A}{\ch}}$ and $\coefsum{y}{\subsets{A}{\ch}}$  would include $\coef_A$, which would guarantee $\coefsum{x}{\subsets{A}{\ch}} \neq \coefsum{y}{\subsets{A}{\ch}}$. Indeed, $\coef_A$ is the unique smallest  summand, so  no combination of other terms could properly compensate for the small difference.  By Corollary \ref{supersets}, the matching parity of $x_{A}$ and $y_{A}$ means that $\coefsum{x_{A}}{\supersets{A}{\ch}} = \coefsum{y_{A}}{\supersets{A}{\ch}}$. Since $\supersets{A}{\ch}$ is a nested chain of subsets, we conclude that 
$\coefsum{x}{\supersets{A}{\ch}} = \coefsum{y}{\supersets{A}{\ch}}$ as well. Our assumption that
 $\coefsum{x}{\subsets{A}{\ch}} = \coefsum{y}{\subsets{A}{\ch}}$ means that 
$\coefsum{x}{\ch} = \coefsum{y}{\ch}$, so equation \eqref{eq:case3} holds.

We now calculate $\coefsumlink{x}{\linkage} - \coefsumlink{y}{\linkage}.$ We can ignore the sets in  $\disjointsets{A}{\linkage}$ since their contributions only depend on the shared partial outcome $u_{[n]-A}$. Since $x_{A}$ and $y_{A}$ have the same parity,  Corollary \ref{supersets} guarantees that  $\coefsumlink{x_{A}}{\supersets{A}{\linkage}}$ = $\coefsumlink{y_{A}}{\supersets{A}{\linkage}}$.
Therefore,
$$
\entry{\setvec{\ch}}{x} - \entry{\setvec{\ch}}{y} 
= \coefsumlink{x_{A}}{\subsets{A}{\linkage}} - \coefsumlink{y_{A}}{\subsets{A}{\linkage}}
$$
and this value is independent of the choice of partial outcome $u_{[n]-A}$.
This proves that every $A \in \ch$ is separable in the partial order induced by $\setvec{\ch}$. 
\qed

%
%
\subsubsection{Proof of Lemma \ref{lem:notcon}}

We will show that if a set $B \subset [n]$ is not $\ch$-constructible, then the set $B$ is not separable on $\setvec{\ch}$.
We assemble two partial outcomes $x_B \neq y_B$  so that the preference between $x=x_Bu_{\comp{B}}$ and $y=y_Bu_{\comp{B}}$ depends on the choice of $u_{\comp{B}}$.

Let $B \subsetneq [n]$ be a nonempty set that is not $\ch$-constructible. Let $K=\construct{B}{\ch} \subsetneq B$ be the $\ch$-construct of $B$.
(Note that we might have $K = \emptyset$; the set $B=\{1\}$ in the Hasse diagram of Figure \ref{fig:char-to-vec-example} is one such example.) 
Consider the set $\mathcal{F} = \{ A \in  \ch \mid A \cap (B-K) \neq \emptyset \}$. Note that $\mathcal{F} \neq \emptyset$ since $[n] \in \mathcal{F}$. Pick a minimal set $A' \in \mathcal{F}$, meaning that $A'$ does not contain any other member of $\mathcal{F}$. 
 Observe that $A' - B \neq \emptyset$ by our choice of $K$.
Let $a_1 \in A'\cap (B-K)$ and let $a_2 \in A' - B$.
By the structure of our tree character and the minimality of $A$, if $a_1 \in A$ for some $A \in \ch$, then $A' \subset A$.

We have the freedom to construct $x=x_Bw_{\comp{B}}$ and $y=y_Bw_{\comp{B}}$ any way we like. Using the notation introduced in section \ref{sec:basis}, and recalling that $a_1 \in B$ and $a_2 \notin B$, we take
$$
\begin{array}{rclcrcl}
   x_B &=&  z_{*}, & \quad&  u_{\comp{B}} &=&  z_{*},  \\
   y_B &=&  1_{a_1} z_{*}, && v_{\comp{B}} &=&  1_{a_2}z_{*}.
\end{array}
$$

Let $S \in \ch \cup \linkage$ such that $a_1 \notin S$. By Corollary \ref{disjointed}, 
we have the equality $\entry{\setvec{S}}{x_Bw_{\comp{B}}} = \entry{\setvec{S}}{y_Bw_{\comp{B}}}$.
This means that for any partial outcome $w_{\comp{B}}$, the difference between $\entry{\setvec{\ch}}{x_Bw_{\comp{B}}}$ and $\entry{\setvec{\ch}}{y_Bw_{\comp{B}}}$ must be caused by coefficients of sets $S \in \ch \cup \linkage$ with $a_1 \in S$. As noted above,  $A' \subset S$ because $A'$ is the minimal set in $\mathcal{F}$ that contains $a_1$. Therefore $a_2 \in S$ as well.

The preference order of $x_Bw_{\comp{B}}$ and $y_Bw_{\comp{B}}$ is determined by the sign of $\entry{\setvec{\ch}}{{x_Bw_{\comp{B}}}} - \entry{\setvec{\ch}}{{y_Bw_{\comp{B}}}} $. We have
\begin{align*}
\coefsum{x_Bw_{\comp{B}}}{\ch} - \coefsum{y_Bw_{\comp{B}}}{\ch}   &= 
  \coef_{A'} \left(\entry{\setvec{A'}}{{x_Bw_{\comp{B}}}} -   \entry{\setvec{A'}}{{y_Bw_{\comp{B}}}} \right) \\
  & \qquad +  
\!\!\! \sum_{A \in \supersets{A'}{\ch}} \!\!\!\!\! \coef_A \left(\entry{\setvec{A}}{{x_Bw_{\comp{B}}}}  -\entry{\setvec{A}}{{y_Bw_{\comp{B}}}} \right).
\end{align*}
We  use equation \eqref{eq:within1} to bound the impact of the sublink coefficients.
 Observe that $x_Bu_{\comp{B}}$ is even in $A'$ while $y_Bu_{\comp{B}}$ is odd in $A'$. Therefore,
$$
\entry{\setvec{\ch}}{{x_Bu_{\comp{B}}}} - \entry{\setvec{\ch}}{{y_Bu_{\comp{B}}}} > \coef_{A'} - \!\!\!\!  \sum_{A \in \supersets{A'}{\ch}} \!\!\!\!\! \coef_A  \,\,\,  - 1\geq 0
$$
by equation \eqref{gparents}. Similarly,  $x_Bv_{\comp{B}}$ is odd in $A'$ while $y_Bv_{\comp{B}}$ is even in $A'$, so
$$
\entry{\setvec{\ch}}{{x_Bv_{\comp{B}}}} - \entry{\setvec{\ch}}{{y_Bv_{\comp{B}}}} < -\coef_{A'} + \!\!\!\!  \sum_{A \in \supersets{A'}{\ch}} \!\!\!\!\! \coef_A  +1 \,\,\, \leq 0.
$$
We have shown that  $x_Bu_{\comp{B}} \succ y_Bu_{\comp{B}}$, while $x_Bv_{\comp{B}} \prec y_Bv_{\comp{B}}$.
Therefore, $B$ is nonseparable.
\qed

%
%
\subsubsection{Proof of Lemma \ref{lem:siblings}}

 Consider a constructible set  $B\notin\ch$ where $\Upsilon_{\ch}(B)  = \{ B_1, B_2, \ldots, B_k \} \subset \ch$ and all of the $B_i$ are children of $P \in \ch$. Note that $k \geq 2$ since
 $B = \construct{B}{\ch} = \cup_{i=1}^k B_i$ while $B\notin\ch$.   We prove that $B$ is nonseparable.

 Let the remaining children of $P$ be $A_1, \ldots, A_{\ell}$  where $\ell \geq 1$ (because $B \subsetneq P$) and perhaps one $A_j$ is a ghost child of $P$. 
For $i=1,2$, let $b_i \in B_i - \construct{B_i}{\ch},$ and for $1 \leq j \leq \ell$, pick any $a_j \in A_j$. 
Consider the partial outcomes  $x_B= 1_{b_2} 0_*$ and $y_B= 1_{b_1}0_*$ on $B$ and the partial outcomes
$u_{\comp{B}}=0_*$ and  $v_{\comp{B}}=1_{a_{1}}1_{a_{2}} \cdots 1_{a_{\ell}}0_*$. We will show that our preference between $x_Bu_{\comp{B}}$ and $y_Bu_{\comp{B}}$ is the opposite of our preference between $x_Bv_{\comp{B}}$ and $y_Bv_{\comp{B}}$. 

Because $P$ is the parent of both $B_1,B_2$, the partial outcomes $x$ and $y$ have the same parity in every set in $\ch - \{ B_1, B_2 \}$. In addition, $B_1$ and $B_2$ are in the same generation, so for any partial outcome $w_{\comp{B}}$ on $\comp{B}$, we have
$$
\coefsum{x_Bw_{\comp{B}}}{\ch} = \coefsum{y_Bw_{\comp{B}}}{\ch},
$$
which means that 
$$
\entry{\setvec{\ch}}{x_Bw_{\comp{B}}} - \entry{\setvec{\ch}}{y_Bw_{\comp{B}}} =
\coefsumlink{x_Bw_{\comp{B}}}{\linkage} - \coefsumlink{y_Bw_{\comp{B}}}{\linkage}.
$$
The parities of these outcomes agree on all sublinks, except for those of the form
$B_1 \cup A_j$ and $B_2 \cup A_j$.
Our choice of partial outcomes $u$ and $v$ flips the parities of the outcomes in these sets so that
$$
 \coefsumlink{x_Bu_{\comp{B}}}{\linkage} - \coefsumlink{y_Bu_{\comp{B}}}{\linkage} =  \coefsumlink{y_Bv_{\comp{B}}}{\linkage} - \coefsumlink{x_Bv_{\comp{B}}}{\linkage}
$$
and this value is nonzero by Lemma \ref{lem:link-distinct}.
Therefore,  $B$ is not separable. 
\qed

%
%
\subsubsection{Proof of Lemma \ref{lem:notsiblings}}

In this section, we prove the nonseparability of a $\ch$-constructible set  $B \notin \ch$ where at least two sets in $\Upsilon_{\ch}(B)$ have different parents.

For a nonempty set $S$, note that there is at least one  outcome that is even on $S$ (the all-zero outcome $0_*$) and at least one outcome that is odd on $S$ (the indicator outcome $1_s 0_*$ for $s \in S$). 
We start with a lemma that constructs an outcome with a specified behavior on a given chain of subsets in $\ch$. We anticipate that the construction technique of this lemma will be useful beyond its application in proving Lemma 
\ref{lem:notsiblings}.

\begin{lemma}
\label{nesting}
Consider a chain of nested sets $ \emptyset \neq  A_{1} \subsetneq A_{2} \subsetneq \cdots \subsetneq A_{r}.$
For any $T \subset [r]$, there is an outcome $w$ such that for $1 \leq i \leq r$, we have
$\setvec{A_{i}}(w) = 1$ if $i \in T$ and $\setvec{A_{i}}(w) = 0$ if $i \in [r]-T$.
\end{lemma}

\begin{proof}
For $1 \leq i \leq r$, let $a_i \in A_i - A_{i-1}$, where we take $A_0 = \emptyset$. We recursively construct an outcome of the form
$$
w = w_{a_1} w_{a_2} \cdots w_{a_r} 0_*.
$$
If $1 \in T$ then take $w_{a_1}=0$. If $1 \notin T$ then take $w_{a_1}=1$. For $2 \leq i \leq r$, take $w_{a_i}$ to be 0 or 1 depending on whether $w_{a_1} \cdots w_{a_{i-1}}$ is even or odd and whether $i \in T$.  
\qed \end{proof}

As an example,
consider the nested chain $A_1 \subsetneq A_2 \subsetneq A_3 \subsetneq A_4$ given by
$$
  \{ 1,2 \} \subsetneq \{1,2,3\} \subsetneq \{1,2,3,4,5,6\} \subsetneq \{1,2,3,4,5,6,7,8\}
 $$
and let $T=\{1,3\}.$ 
We choose $(a_1, a_2, a_3, a_4)=(1,3,4,7)$ so that $a_i \in A_i - A_{i-1}$, where we set $A_0=\emptyset$.
The outcome $0_{1}1_{3}0_{4}1_{7} 0_* = 00100010$ is even on $A_1,A_3$ and odd on $A_2, A_4$.

We now begin the proof of Lemma \ref{lem:notsiblings} in earnest.
 Consider a constructible set  $B \notin \ch$ with $\Upsilon_{\ch}(B)  = \{ B_1, B_2, \ldots, B_k \} \subset \ch$ where  $B_1$ and $B_2$ are not siblings in  $\ch$.   For $i=1,2$ let $K_i = \construct{B_i}{\ch} \subsetneq B_i$ and  $b_i \in B_i - K_i$.

Consider the  partial outcomes
$x_B  =  1_{b_2}0_*$ and $y_B   = 1_{b_1} 0_*$ on $B$.
We must track the parity of the partial outcomes $x_B$ and $y_B$ with respect to  $S \subset \ch$.
Observe that $x_B$ is even in $B_1$ and odd in $B_2$, while $y_B$ is odd in $B_1$ and even in $B_2$. More generally, if $B_1 \subset S$ but $S \cap B_2 = \emptyset$ then $x_B$ is even in $S$ while $y_B$ is odd in $S$, and if $B_2 \subset S$ but $S \cap B_1 = \emptyset$ then $x_B$ is odd in $S$ while $y_B$ is even in $S$.
Next, we note that if $S \cap B_1 = \emptyset$ and $S \cap B_2 = \emptyset$, then  both $x_B$ and $y_B$ are zero (hence even) in $S$. Finally, if $B_1 \cup B_2 \subset S$, then  $x_B$ and $y_B$ have the same parity in $S$ by Corollary \ref{supersets}.

Let $g_1=g(B_1)$ and $g_2=g(B_2)$ denote the generations of $B_1$ and $B_2$, respectively.
Let  $P$ be the first shared ancestor of $B_1$ and $B_2$. We may assume that $g_1 \geq g_2 > g$ and that $\supersets{B_1}{\ch} - \supersets{B_2}{\ch} \neq \emptyset$  because $B_1$ and $B_2$ are not siblings.
We are ready to use Lemma \ref{nesting} to construct $u_{\comp{B}}$ and $v_{\comp{B}}$ that  change the preference between the partial outcomes $x_B$ and $y_B$. There are two cases, depending on shared ancestry of $B_1$ and $B_2$.

We may assume that either $g_1 > g_2 = g+1$ or $g_1 \geq g_2 > g+1$. For $i=1,2$, let $Q_i \in \ch$ be the child of $P$ that contains $B_i$. (If $B_2$ is a child of $P$ then $B_2=Q_2$.) Let $R_1$ be the parent of $B_1$ (so that when $g_1=g+2$, we have $R_1=Q_1$).

Let $r_1 \in R_1 -B_1$. Observe that  that  $r_1 \notin B_2$, while $r_1,b_1, b_2 \in P$. We take $u_{\comp{B}}=0_*$ to be the all-zero outcome and $v_{\comp{B}}=1_{r_1} z_{*}$ to be the indicator outcome on $r_1$. 
Let  $S \in \supersets{B_1}{\ch} \cap \strictsubsets{P}{\ch}$ and $T \in \supersets{B_2}{\ch} \cap \strictsubsets{P}{\ch}$. Note that the ancestry $\supersets{B_2}{\ch} \cap \strictsubsets{P}{\ch} = \emptyset$ when $B_2$ is a child of $P$; in this case, statements below concerning $T$ hold vacuously.

 First, we consider outcomes that end in $u_{\comp{B}}$.
 Observe that $x_B u_{\comp{B}} = 1_{b_2} 0_*$ is even in $B_1$, but odd in $B_2$. Meanwhile, $y_B u_{\comp{B}} = 1_{b_1} 0_*$ is odd in $B_1$, but even in $B_2$. 
The outcome  $x_B u_{\comp{B}}$ is even in $S$ (odd in $T$) while $y_B u_{\comp{B}}$ is odd in $S$ (even in $T$). 
Therefore
\begin{align*}
\coefsum{x_Bu_{\comp{B}}}{\ch} - \coefsum{y_Bu_{\comp{B}}}{\ch} 
&=
\base^{g_1} + \sum_{i=g+1}^{g_1-1} \base^i - \left(\base^{g_2} + \sum_{i=g+1}^{g_2-1} 2^i \right) \\
& = \sum_{i=g_2+1}^{g_1} \base^i > 1.
\end{align*}

Next, we consider outcomes that end in $v_{\comp{B}}$.
Observe that $x_B v_{\comp{B}} = 1_{b_2} 1_{r_1} 0_*$ is even in $B_1$ but odd in $B_2$, while $y_B u_{\comp{B}} = 1_{b_1} 1_{r_1} 0_*$ is odd in $B_1$ but even in  $B_2$. Turning to the ancestry to $P$, we see that
$x_B v_{\comp{B}}$ is odd in $S$ (even in $T$) while $y_B v_{\comp{B}}$ is even in $S$ (odd in $T$). Therefore
\begin{align*}
\coefsum{x_Bv_{\comp{B}}}{\ch} - \coefsum{y_Bv_{\comp{B}}}{\ch} 
&=
\base^{g_1} + \sum_{i=g+1}^{g_2-1} \base^i - \left(2^{g_2} + \sum_{i=g+1}^{g_1-1} \base^i \right) \\
&=  \base^{g_1}  - \left(\base^{g_2} + \sum_{i=g_2}^{g_1-1} \base^i \right)  < 0
\end{align*}
by Lemma \ref{lem:basesum}.

We have  shown that $x_Bu_{\comp{B}} \succ y_Bu_{\comp{B}}$ and $x_Bv_{\comp{B}} \prec y_Bv_{\comp{B}}$, so the set $B$ is not separable
\qed


\section{Conclusions and Future Work}

\label{chap:conclusion}

The admissibility problem asks which collections of sets correspond to characters of preference orderings. We have introduced the voter basis and used this basis to create preference orderings with desired separability properties. In particular, we have shown that every tree character is admissible. We believe that our tree construction just begins to tap into the potential of the voter basis for character construction, and we are actively working on constructing other families of admissible characters. We also wonder whether the voter basis can provide insight into the class of completely separable preferences. 

The proof herein for showing that $\VV_n$ forms a basis for the preference space $\PP^n$ is short and effective. However, its simplicity hides the deep connection between constructing voter preferences and the symmetries of the hypercube. We are hopeful that the representation theory can provide further insight into the admissibility problem.

\bigskip

{\bf Acknowledgments.}
 We thank Tom Halverson for many insightful conversations and his help in developing the voter basis; Jeremy Martin for suggesting  Lemma \ref{lem:martin}; and Trung Nguyen and Tuyet-Anh Tran for their careful readings of earlier drafts.

\bibliography{refs}        
\bibliographystyle{plain}  

\end{document}